\documentclass[a4paper,10pt]{amsart}

\usepackage[colorlinks=true]{hyperref}

\usepackage{verbatim,latexsym,exscale,amssymb,amsthm,amsmath,amsfonts,mathrsfs,amsrefs}
\usepackage{stmaryrd}
\usepackage[all]{xy}

\xyoption{ps}
\SelectTips{cm}{}
\newdir{ >}{{}*!/-6pt/@{>}}

\numberwithin{equation}{section}
\newcommand{\ie}{\textrm{i.e.}}

\newcommand{\cC}{\mathcal{C}}
\newcommand{\cM}{\mathcal{M}}

\newcommand{\cV}{\mathcal{V}}

\newcommand{\Cor}{\mathrm{Cor}}

\newcommand{\GalMod}{\mathop{\mathsf{GalMod}}}
\newcommand{\Mod}{\mathop{\mathsf{Mod}}}

\newcommand{\Fun}{\mathsf{Fun}}

\newcommand{\Pic}{\mathop{\mathrm{Pic}}}
\newcommand{\Br}{\mathop{\mathrm{Br}}}

\newcommand{\proj}{\mathop{\mathsf{proj}}}

\newcommand{\Mor}{\mathrm{Mor}}

\newcommand{\id}{\mathrm{id}}

\newcommand{\K}{\mathbb{K}}

\newcommand{\op}{\mathrm{op}}

\newcommand{\can}{\mathrm{can}}
\newcommand{\colim}{\mathop{\mathrm{colim}}}
\newcommand{\unit}{1\!\!1}
\newcommand{\obj}{\mathop{\mathrm{ob}}}

\newcommand{\Image}{\mathrm{Im}}
\newcommand{\End}{\mathrm{End}}

\newcommand{\Hom}{\mathrm{Hom}}

\newcommand{\Cat}{\mathsf{Cat}}
\newcommand{\Alg}{\mathsf{Alg}}

\newcommand{\Weq}{\mathrm{Weq}}
\newcommand{\Fib}{\mathrm{Fib}}
\newcommand{\Cof}{\mathrm{Cof}}
\newcommand{\Ho}{\mathrm{Ho}}

\newcommand{\too}{\longrightarrow}

\numberwithin{equation}{section}

\newtheorem{thm}[equation]{Theorem}
\newtheorem{lemma}[equation]{Lemma}
\newtheorem{thm-defi}[equation]{Theorem-Definition}
\newtheorem{prop}[equation]{Proposition}
\newtheorem{cor}[equation]{Corollary}
\newtheorem*{thm*}{Theorem}
\newtheorem*{lemma*}{Lemma}
\newtheorem*{cor*}{Corollary}

\theoremstyle{definition}

\newtheorem{defi}[equation]{Definition}
\newtheorem*{conv*}{Conventions}
\newtheorem*{ack}{Acknowledgements}

\newtheorem*{conventions*}{Conventions}
\newtheorem*{acknowledgements*}{Acknowledgements}

\theoremstyle{remark}

\newtheorem{remark}[equation]{Remark}

\newtheorem{notation}[equation]{Notation}

\newtheorem*{remark*}{Remark}

\begin{document}
\title[A Morita Quillen model and a $\otimes$-categorification of Brauer groups]{A Quillen model for classical Morita theory and a tensor categorification of the Brauer group
}
\author{Ivo Dell'Ambrogio}
\address{Ivo Dell'Ambrogio, Laboratoire Paul Painlev\'e, Universit\'e de Lille~1, Cit\'e Scientifique - B\^atiment M2, F-59655 Villeneuve d'Ascq Cedex, France}
\email{ivo.dellambrogio@math.univ-lille1.fr}
\urladdr{www.math.uni-bielefeld.de/~ambrogio/ 
}

\author{Gon\c calo Tabuada}

\address{Gon{\c c}alo Tabuada, Department of Mathematics, MIT, Cambridge, MA 02139, USA}
\email{tabuada@math.mit.edu}
\urladdr{http://math.mit.edu/~tabuada} 
\thanks{The second named author was partially supported by the NEC Award-2742738.}
\address{}
\email{}

\date{\today}

\begin{abstract}
Let $\K$ be a commutative ring. 
In this article we construct a well-behaved symmetric monoidal Quillen model structure on the category of small $\K$-categories which enhances classical Morita theory. Making use of it, we then obtain the usual categorification of the Brauer group and of its functoriality. Finally, we prove that the (contravariant) corestriction map for finite Galois extensions also lifts to this categorification.
\end{abstract}


\maketitle

\vskip-\baselineskip
\vskip-\baselineskip
\vskip-\baselineskip

\section{Introduction}\label{sec:intro}
Let $\K$ be a commutative ring. In this article we address the following questions:
\medbreak
\noindent {\bf Question A:} Can the classical Morita theory of $\K$-algebras be viewed as a ``homotopy theory''?

\noindent  {\bf Question B:} Does the Brauer group of $\K$ admit a natural ``tensor categorification''? 

\subsection*{Morita theory}
The classical Morita theory~\cite{Morita} defines an equivalence relation between $\K$-algebras. Namely, two $\K$-algebras $R$ and $S$ are called {\em Morita equivalent} if they have equivalent categories of representations, or equivalently, if there exist bimodules ${}_RM_S$ and ${}_SN_R$ and bimodule isomorphisms ${}_RM\otimes_S N_R \simeq R$ and ${}_SN \otimes_R M_S \simeq S$. In order to view this classical notion from a homotopical perspective we will consider $\K$-categories, \ie\ categories enriched over $\K$-modules; see~\S\ref{sec:preliminaries}. Every $\K$-algebra is naturally a $\K$-category with a single object and all the usual notions (module, bimodule, tensor product, etc.) extend automatically to this larger setting. In particular, the category $\Cat_\K$ of (small) $\K$-categories is closed symmetric monoidal. Our answer to {\bf Question A} is then the following:
\begin{thm}\label{thm:main1}
The category $\Cat_\K$ carries a combinatorial symmetric monoidal Quillen model structure. The weak equivalences are the {\em Morita equivalences}, \ie\, the $\K$-linear functors $A\to B$ inducing an equivalence $\smash{\Mod A \stackrel{\sim}{\to}\Mod B}$ on the module categories (Definition~\ref{defi:Morita_eq}), and the cofibrations are the $\K$-linear functors that are injective on objects. Moreover, two $\K$-algebras become isomorphic in the associated homotopy category $\Ho(\Cat_\K):= \Cat_\K [\{\textrm{Morita equivalences}\}^{-1}]$ if and only if they are Morita equivalent.
\end{thm}
\subsection*{Brauer group}
Based on the foundational work of Brauer and others, Auslander and Goldman \cite{AG} introduced the Brauer group $\Br(\K)$ of a commutative ring $\K$ in their study of separable algebras. 
Concretely, $\Br(\K)$ consists of the Morita equivalence classes of Azumaya $\K$-algebras with addition induced by the tensor product. 
Any homomorphism $\K \to \mathbb{L}$ of commutative rings gives naturally rise, by extension of scalars, to a map $r_{\mathbb L/\K}\colon \Br(\K) \to \Br(\mathbb{L})$. When $\mathbb{L}$ is a finite Galois extension of a field~$\mathbb{K}$, one moreover  has a well-defined corestriction (or transfer, or Weil restriction) map $c_{\mathbb{L}/\K}\colon \Br(\mathbb{L}) \to \Br(\K)$ in the opposite direction; see~\cite{draxl}.

Conceptually, the Brauer group is known to be the group of $\otimes$-invertible Morita equivalence classes of $\K$-algebras (or small $\K$-categories), i.e.\ it is a \emph{Picard group}. 
We may think of the three illustrious mathematicians as being related by the equation:
$$ \textrm{Brauer}(\K) = \mathrm{Picard}\big((\K\textrm{-algebras},\otimes_\K ) / \textrm{Morita} \big)\,.$$
Let us make this precise. Recall that a {\em categorification} of a group~$H$ consists of an additive\footnote{In one common variant; one may instead want to use exact or triangulated categories.} category $\cC_H$ with the property that its Grothendieck group $K_0(\cC_H)$ identifies with $H$. 
In this vein we call {\em tensor categorification of~$H$}  (or \emph{$\otimes$-categorification}) any symmetric monoidal category $\cC_H^\otimes$ with the property that its Picard group $\mathrm{Pic}(\cC_H^\otimes)$ identifies with $H$. Our answer to {\bf Question B} is the~following:
\begin{thm}\label{thm:main2}
\begin{itemize}
\item[(i)] The symmetric monoidal category $\Ho(\Cat_\K)^\otimes$ of $\otimes$-invertible objects in the homotopy category $\Ho(\Cat_\K)$ is a $\otimes$-categorification of $\Br(\K)$;
\item[(ii)] Given any ring homomorphism $\K \to \mathbb{L}$, we have a base-change $\otimes$-functor $-\otimes_\K\mathbb{L}\colon \Ho(\Cat_\K)^\otimes \to \Ho(\Cat_{\mathbb{L}})^\otimes$ such that $\mathrm{Pic}(-\otimes_\K \mathbb{L})\simeq r_{\mathbb L/\K}$;
\item[(iii)] When $\mathbb L$ is a finite Galois extension of a field~$\K$, we have a corestriction $\otimes$-functor $C_{\mathbb{L}/\K}\colon \Ho(\Cat_{\mathbb{L}})^\otimes \to \Ho(\Cat_\K)^\otimes$ such that $\mathrm{Pic}(C_{\mathbb{L}/\K})\simeq c_{\mathbb{L/\K}}$.
\end{itemize}
\end{thm}
The base-change and corestriction functors of parts~(ii) and~(iii) are already defined at the level of the categories $\Cat_\K$ and $\Cat_{\mathbb{L}}$; consult \S\ref{sec:tensor-categorification}-\ref{sec:corestriction} for details.
Note that Theorem~\ref{thm:main2} not only conceptually $\otimes$-categorifies the Brauer group but its functorial behavior too. 
\begin{remark}
Of course, the above ideas on the Brauer group have been around for quite some time. They have also been extended from rings to more general objects such as dg-algebras, and to more sophisticated categorical and homotopical frameworks; see for instance 
\cite{Ant-Gep}, 
\cite{baez:week209}, 
\cite{B-Vitale}, 
\cite{Duskin}, 
\cite{Johnson},
\cite{M-Vitale}, 
\cite{toen:derived_Morita}, 
\cite{toen:azumaya},~\ldots 
The goal of this article is to look back at the original setting of $\K$-algebras and show that one can add a very nice model structure to the classical picture (this model is adapted from~\cite{CcatMorita}, where C*-categories were considered instead of $\K$-categories). 
The way corestriction fits into the story also seems to have gone unrecorded so far, in particular how it lifts to the categorification as a Morita invariant $\otimes$-functor. 
We otherwise don't claim much originality here, and encourage the reader to regard this article as a review (with full proofs) of the notion of Brauer group from the point of view of homotopical algebra; or alternatively, as a detailed case study of a certain simple but rich Quillen model structure.
\end{remark}

\begin{ack}
The first named author wishes to thank Greg Stevenson for some help with the tensor invertible objects, and the mathematics department at MIT for the pleasant stay during which this work was conceived. The second named author wishes to thank Thomas Bitoun for useful discussions. Both authors would like to thank Ben Antieau for comments and references. Finally, the authors are very grateful to the anonymous referee for all his/her comments, suggestions and references that greatly allowed the improvement of the article.
\end{ack}

\section{$\K$-linear preliminaries} \label{sec:preliminaries}
Let $\K$ be a commutative and associative ring with unit, that we fix throughout. 
We denote by $\Mod \K$ the category of $\K$-modules, and equip it with the usual closed symmetric monoidal structure~$\otimes_\K$.
A \emph{$\K$-category} $A$ is a category where each Hom set $A(x,y)$ 
comes equipped with the structure of a $\K$-module such that composition is $\K$-bilinear. 
A \emph{$\K$-linear functor} $F\colon A\to B$, or $\K$-functor, is a functor  such that the structure maps $F\colon A(x,y)\to B(Fx,Fy)$ are $\K$-linear for all objects $x,y\in \obj A$.
For $\K$-categories $A$ and $B$, with $A$ small, we denote by $\Fun_\K(A,B)$ the $\K$-category of $\K$-linear functors $A\to B$ and natural transformations between them; $\Mod A:=\Fun_\K(A^\op, \Mod \K)$ will denote the category of right $A$-modules. 
The usual Yoneda embedding $h_A\colon A\to \Mod A$ sending $x\in \obj A$ to $h_A(x)=A(-,x)$ is an example of a $\K$-linear functor. An $A$-module is called \emph{representable} if it is isomorphic to one of the form $h_A(x)$.
\begin{defi}
\label{defi:Morita_eq}
A \emph{Morita equivalence} is a $\K$-linear functor $F\colon A\to B$ with the property that the left Kan extension  $\Mod A \to\Mod B$ of $h_B\circ F\colon A\to B\to \Mod B$ along $h_A\colon A\to \Mod A$ (where $h_A,h_B$ are the Yoneda embeddings) is an equivalence.
\end{defi}
We write $\Cat_\K$ for the category of all \emph{small} $\K$-categories and $\K$-linear functors between them. Note that $\Cat_\K$ contains the category $\Alg_\K$ of all $\K$-algebras as a full subcategory: simply consider every $\K$-algebra $A$ as a $\K$-category with a single object having $A$ as its endomorphism algebra.  
\subsection*{Monoidal structure}
The category $\Cat_\K$ inherits from $\Mod \K$ a closed symmetric monoidal structure: the tensor product $A\otimes B$ of two $\K$-categories has object set $\obj (A\otimes B)=\obj A\times \obj B$, Hom $\K$-modules $(A\otimes B)((x,y),(x',y'))=A(x,x')\otimes_\K B(y,y')$, and the evident induced composition; the tensor unit object is~$\K$, and the internal Hom is provided by $\Fun_\K(-,-)$.
\subsection*{Direct sums}

Let $A\in \Cat_\K$.
The \emph{additive hull} $A_\oplus$ of $A$ is the small $\K$-category defined as follows: its objects are formal words $x_1\cdots x_n$ (also written $x_1\oplus\cdots\oplus x_n$) on the set $\obj(A)$ and the Hom $\K$-modules are the spaces of matrices, written: 
\[
A(x_1\cdots x_n, y_1\cdots y_m) := \bigoplus_{ \substack{j=1,\ldots ,n \\ i=1,\ldots, m}} A(x_j, y_i)  \; \ni \;[a_{ij}]\,.
\]
Composition is the usual matrix multiplication, $[b_{ij}]\circ [a_{ij}] = [\sum_k b_{ik} a_{kj} ]$. There is a canonical fully faithful functor $ \sigma_A\colon A \too A_\oplus $, $x \mapsto x$.
Given a $\K$-linear functor $F\colon A\to B$, we define a new $\K$-linear functor $F_\oplus \colon A_\oplus \to B_\oplus$ by setting
\[
F_\oplus(x_1\cdots x_n):= F(x_1) \cdots F(x_n)
\quad \textrm{ and } \quad
F([a_{ij}]):= [F(a_{ij})]
\]
for all objects $x_1\cdots x_n\in \obj(A_\oplus)$ and arrows $[a_{ij}]\in A_\oplus$. We obtain in this way a well-defined additive hull functor $(-)_\oplus\colon \Cat_\K \to\Cat_\K$.
\subsection*{Retracts}
Let $A\in \Cat_\K$. We denote by $A^\natural$ the \emph{idempotent completion} (a.k.a.\ pseudo-abelian or Karoubian envelope) of $A$. Its objects are the pairs $(x,e)$ where $x\in \obj(A)$ and $e=e^2$ is an idempotent endomorphism of~$x$. Its Hom $\K$-modules are $A^\natural ((x,e),(y,f))= f A(x,y) e$, and the composition is induced by that of $A$. This defines a functor $(-)^\natural\colon \Cat_\K\to \Cat_\K$ which comes equipped with a natural embedding $\tau_A\colon A\to A^\natural$, $x\mapsto (x,1_x)$.
\subsection*{Saturation} 
Let $A\in \Cat_\K$. We say that $A$ is \emph{saturated} if it admits all finite direct sums and if all idempotents split.
We define the \emph{saturation} of $A$ to be $A_\oplus^\natural:=(A_\oplus)^\natural$.  Note that $A_\oplus^\natural$ is indeed always saturated. We thus obtain a \emph{saturation functor} $\smash{(-)_\oplus^\natural\colon \Cat_\K\to \Cat_\K}$ and a natural $\K$-linear embedding $\smash{\iota_A:=\tau_{A_\oplus}\circ \sigma_A\colon A\to A_\oplus^\natural}$.
Let $X\subseteq A$ be a (full) subcategory of some~$A\in \Cat_\K$. We say that $X$ \emph{additively generates $A$}, or that $\obj(X)$ is a set of \emph{additive generators for $A$}, if the smallest (full) subcategory of $A$ containing $X$ and closed under taking direct sums and retracts is $A$ itself. We leave the easy proof of the next two lemmas as an exercise for the reader.
\begin{lemma} \label{lemma:first_char}
If $F\colon A\to B$ is any $\K$-linear functor, then $F^\natural_\oplus$ is an equivalence if and only if $F$ is fully faithful and the image of $\iota_B\circ F$ additively generates~$B^\natural_\oplus$.
\qed
\end{lemma} 
\begin{lemma} \label{lemma:unique_extension}
Let $D$ be a (not necessarily small) saturated $\K$-category. If a $\K$-functor $F\colon A\to B$ is fully faithful and its image additively generates~$B$, then every $\K$-linear functor $G\colon A\to D$ extends,  uniquely up to isomorphism, along $F$ to a $\K$-linear functor $B\to D$.
\qed
\end{lemma}
\begin{remark}
If in the definition of a $\K$-category $A$ one replaces $\Mod \K$ by the associated category of cochain complexes of $\K$-modules, one obtains the notion of {\em differential graded (=dg) category} $\mathcal{A}$; see \cite{keller:dg}. We warn the reader that in this setting the notion of saturated (or equivalently smooth and proper) is very different from the above one on $\K$-categories.
\end{remark}

\section{Proof of Theorem~\ref{thm:main1}} \label{sec:can}
The construction of the Quillen model structure on $\Cat_\K$ is divided into two steps. First we construct a well-behaved ``canonical'' Quillen model structure on $\Cat_\K$; see Theorem~\ref{thm:can}. Then we localize it in order to obtain the desired Morita model structure; see Definition~\ref{defi:Mor_model}.
\subsection*{Canonical model structure}
Note that we have a natural adjunction
\begin{equation} 
\label{Quillen_pair_can}
\xymatrix{\Cat_\K \ar@/^2ex/[d]^{[-]} \\ 
\Cat \ar@/^2ex/[u]^{\mathrm F_\K} }
\end{equation}
where $[-]$ is the underlying category functor (which forgets the $\K$-linear structure) and $\mathrm F_\K$ is the free $\K$-category functor, given by the following construction: for a small category~$C$, let $\mathrm F_\K C$ be the $\K$-category with the same objects as~$C$, with Hom $\K$-modules given by $\mathrm{F}_{\K} C(x,y)= \coprod_{C(x,y)} \K$, and with composition induced by that of~$C$. 
Recall, e.g.\ from \cite{rezk:folk}, the definition of the well-known \emph{canonical} (or \emph{folk}) model structure on the category $\Cat$ of small categories. It consists of the following three classes of functors:
\begin{align*}
\Weq_\can &= \{\textrm{(ordinary) equivalences of categories} \} \\
\Cof_\can &= \{\textrm{functors } F\colon A\to B \textrm{ such that } \obj(F)\colon \obj(A)\to \obj(B)\textrm{ is injective}\} \\
\Fib_\can & = \{\textrm{functors } F \textrm{ allowing the lift of isomorphisms of the form } Fx\stackrel{\sim}{\to} y \}\,.
\end{align*}
The canonical model is cofibrantly generated, with the following sets of generating cofibrations and generating trivial cofibrations:
\begin{eqnarray*}
I_\can = \{ \emptyset \to \bullet \,,\,  \bullet \sqcup\, \bullet \to \mathbf1 \,,\,  P\to \mathbf1 \}  &&
J_\can = \{0\colon \bullet\to \mathbf I\}\,.
\end{eqnarray*}
Same explanations are in order. Here $\emptyset$ denotes the initial (empty) category and $\bullet$ the final category (consisting of precisely one object and its identity arrow). The categories $\mathbf1$, $P$ and~$\mathbf I$ all have precisely two objects $0$ and $1$, and, respectively, one non-identity arrow $0\to 1$, a pair of distinct arrows $0\rightrightarrows 1$, and one isomorphism $u\colon 0\stackrel{\sim}{\to} 1$. The functors are the evident ones; in particular, $\bullet \sqcup \bullet \to \mathbf1$ is the inclusion of the two endpoints, and $0\colon \bullet \to \mathbf I$ is the inclusion of~$0$.

\begin{lemma}
\label{lemma:pushout_J_can}
For any $\K$-functor $F\colon A\to B$, consider the following pushout square
 \begin{equation} \label{eq:special_po}
\xymatrix{
A\otimes \mathrm F_\K(\bullet) = A\ar@{}[dr]|{\lrcorner} \ar[d]_{A\otimes \mathrm F_\K(0)} \ar[r]^-F & B \ar[d]^G \\
A\otimes \mathrm F_\K(\mathbf I) \ar[r]^-{H} & \tilde B
} 
 \end{equation}
 in $\Cat_\K$. Then the $\K$-functor $G$ is a $\K$-linear equivalence and the object-function of $H$ is injective on the subset $\{(x,1)\mid x\in \obj A\}\subset \obj(A\otimes \mathrm F_\K(\mathbf I))$.
\end{lemma}

\begin{proof}
To prove the claims it suffices to give an explicit description of~$\tilde B$, $G$ and~$H$ with the required properties.
Let $\tilde B$ be the category with object set $\obj(\tilde B)=\obj(B) \sqcup \obj(A)$ and Hom $\K$-modules given by
\[
\tilde B(x,y):=
\left\{
\begin{array}{ll}
B(Fx,Fy) &  \textrm{if } x,y\in \obj(A) \\
B(Fx,y) & \textrm{if } x\in \obj(A), y\in \obj(B) \\
B(x,Fy) & \textrm{if } x\in \obj(B), y\in \obj(A)  \\
B(x,y) & \textrm{if } x,y\in \obj(B).
\end{array}
\right.
\]
The composition in $\tilde B$ is induced by that of~$B$ in the evident way, and there is an obvious fully faithful inclusion $G \colon B\to \tilde B$ defined by $x\mapsto x$ $(x\in \obj B)$. 
Moreover $G$ is essentially surjective, because for any ``new'' object $x\in \obj(A)\subset \obj (\tilde B)$ the arrow $1_{Fx}\in B(Fx,Fx)= \tilde B(x,Fx)= \tilde B(x,GFx)$ defines an isomorphism in $\tilde B$ between $x$ and an object in the image of~$G$; thus $G$ is a $\K$-equivalence.
There is also a functor $H\colon A\otimes \mathrm F_\K(\mathbf I)\to \tilde B$ defined on objects by $(x,0)\mapsto Fx$ and $(x,1)\mapsto x$ and on arrows by the formula~$H(f\otimes 1_0)=H(f\otimes 1_1)=H(f\otimes u)=F(f)$.
Clearly $H$ satisfies the required injectivity, and the resulting square \eqref{eq:special_po} is commutative. It only remains to verify the pushout property. 
Consider a diagram of $\K$-functors
\[
\xymatrix{
A \ar[d]_{A \otimes \mathrm F_\K (0)} \ar[r]^-{F}  & B \ar[d]^G \ar@/^3ex/[ddr]^{T_0}  \\
A\otimes  \mathrm F_\K (\mathbf I) \ar@/_3ex/[drr]_{T_1} \ar[r]^-{H} & \tilde B  \ar@{..>}[dr]|T & \\
&& C \,.
}
\]
such that $T_0F=T_1(A \otimes \mathrm F_\K(0))$.
In order to complete this to a commutative diagram, the $\K$-functor $T\colon \tilde B\to C$ must be defined on objects by 
\[
Tx :=
\left\{
\begin{array}{ll}
T_1(x,1) & \textrm{if } x\in \obj(A) \\
T_0x & \textrm{if } x\in \obj(B)
\end{array}
\right.
\]
and on arrows $f\in \tilde B(x,y)$ by
\[
T(f):=
\left\{
\begin{array}{ll}
T_1(1_y\otimes u)\circ  T_0(f)\circ T_1(1_x\otimes u^{-1}) & \textrm{if }  x,y\in \obj(A) \\
T_0(f)\circ T_1(1_x\otimes u^{-1}) & \textrm{if }  x\in \obj(A), y\in \obj(B) \\
T_1(1_y\otimes u)\circ T_0(f) & \textrm{if }   x\in \obj(B), y\in \obj(A)  \\
T_0 (f) & \textrm{if }  x,y\in \obj(B).
\end{array}
\right.
\]
It is straightforward (though mildly tedious) to verify that $T$ is well-defined, makes the diagram commute, and is the unique such $\K$-functor.
This shows that we have indeed constructed the required pushout.
\end{proof}

\begin{thm}
\label{thm:can}
 The category $\Cat_\K$ carries a Quillen model structure where the weak equivalences are the $\K$-equivalences, the fibrations are the $\K$-linear functors $F$ such that $[F]$ is a fibration in~$\Cat$, and the cofibrations are the $\K$-linear functors that are injective on objects. In particular every object is fibrant and cofibrant and the adjunction \eqref{Quillen_pair_can} becomes a Quillen pair. Moreover, this Quillen model structure satisfies the following properties:
\begin{enumerate}
\item[(i)] It is cofibrantly generated and we may take the sets $\mathrm F_\K (I_\can)$ and $\mathrm F_\K (J_\can)$ as the generating cofibrations and trivial cofibrations.
\item[(ii)] It is symmetric monoidal in the sense of Mark Hovey \textup(see \cite{hovey:model}*{ch.\,4}\textup);
\item[(iii)] Is is combinatorial in the sense of Jeff Smith \textup(see~\cite{beke:sheafifiable}\textup);
\item[(iv)] Every map $F\colon A\to B$ admits a natural ``mapping cylinder'' factorization (as a cofibration $J$ followed by a trivial fibration~$Q$)
\[
\xymatrix{ 
F:A\, \ar@{>->}[r]^-J & \tilde B \ar@{->>}[r]^-Q_-\sim & B\,,
}
\]
where $\tilde B$ is the pushout $(A \otimes \mathrm F_\K \mathbf (I)) \sqcup_A B$ of $A\otimes \mathrm F_\K(0) \colon A=A\otimes \K \to A\otimes \mathrm F_\K \mathbf (I)$ along~$F$ and $J$ is the composite $A=A\otimes \K \stackrel{A\otimes \mathrm F_\K(1)}{\longrightarrow} A\otimes \mathrm F_\K \mathbf (I)\to \tilde B$. The $\K$-linear functor $Q$ is induced by the pushout property of $\tilde B$ by the two $\K$-linear functors $\id_B$ and $F\otimes F' \colon  A\otimes \mathrm F_\K \mathbf (I) \to B\otimes \K =B$, where $F':=\mathrm F_\K(u\mapsto 1_{\K})$. (See \eqref{eq:mapping_cylinder} for a pictorial description.)

\end{enumerate}
\end{thm}

\begin{defi}
We will call the model structure of Theorem \ref{thm:can} the \emph{canonical model structure} on $\Cat_\K$ and we will denote it by $\cM_\can$.
\end{defi}

\begin{remark}
Note that $\cM_\can$ is not cellular in the sense of Hirschhorn \cite{hirschhorn}*{Definition~11.1.1} since not all cofibrations are monomorphism. For instance from the adjunction \eqref{Quillen_pair_can} one observes that the generating cofibration $\mathrm F_\K(P\to \mathbf1)$ is not a monomorphism since it admits two evident (fully faithful) distinct sections $S_1, S_2\colon \mathrm F_\K (\mathbf1)\to \mathrm F_\K (P)$.
\end{remark}

\begin{remark}
The analog of Theorem \ref{thm:can} holds (with the same proof) for $\cV\textrm-\Cat$, the category of small categories enriched over a bicomplete closed symmetric monoidal category~$\cV$, at least if one assumes that the tensor unit $\unit\in \cV$ is a finite object; see \cite{hovey:model}*{2.1.1}. Also, in order for the weak equivalences to coincide with the $\cV$-equivalences one should  assume that $[-]=\Hom_\cV(\unit,-)$ detects $\cV$-equivalences, and in order for the model to be combinatorial one should assume $\cV$ locally presentable. In case~$\mathcal V$ itself comes equipped with a (sufficiently nice) symmetric monoidal model structure, there exist general results for lifting the model structure of the base to the ``canonical'' one on $\cV\textrm-\Cat$, whose equivalences are Hom-wise in~$\cV$; see~\cite{Berger-Moerdijk},~\cite{Lurie}*{A.3}.
\end{remark}

\begin{proof}[Proof of Theorem \ref{thm:can}]
In order to establish the model structure it suffices to check conditions (1) and (2) of Kan's lifting theorem \cite{hirschhorn}*{Theorem 11.3.2}.
Condition (1) follows from the fact that the domains of the maps in $\mathrm F_\K(I_\can)$ and $\mathrm F_\K(J_\can)$ are small objects; note that $\K$ is small (even finite) in~$\Mod \K$. The functor $[-]$ preserves sequential colimits and a $\K$-linear functor $F$ is a $\K$-equivalence in $\Cat_\K$ if and only if $[F]$ is an equivalence in~$\Cat$. Hence, condition (2) follows from Lemma~\ref{lemma:pushout_J_can}. 
 This establishes the model structure with the described weak equivalences and fibrations, and moreover proves~(i).

Since $\Mod \K$, being a Grothendieck category, is locally presentable then so is $\Cat_\K$ by the main result of~\cite{kelly-lack:VCAT}. Hence the model structure on $\Cat_\K$ established above is not only cofibrantly generated  but also combinatorial, as claimed in~(iii).
Denote by $\cC$ the class of $\K$-functors which are injective on objects and by $\Cof$ the class of cofibrations of the model structure. 
Then, the same argument as the one in the proof of \cite{ivo:unitary}*{Lemma 4.10(ii)} shows us that $\cC\subseteq \Cof$.

Before proving the converse inclusion, let us establish~(iv).
For any $F\colon A\to B$, perform the construction described in~(iv) (we identify $A$ with $A\otimes \K$ and $B$ with $B\otimes \K$):
\begin{equation}\label{eq:mapping_cylinder}
\xymatrix{
 & A  \ar[r]^-F \ar[d]_{A\otimes \mathrm F_\K(0)} \ar@{}[dr]|{\lrcorner} &
  B \ar[d]^G \ar@/^2ex/[dr]^{1_B} &
   \\
 A \ar@/_4ex/[rr]_{J} \ar[r]^-{A\otimes 1 } &
  A\otimes \mathrm F_\K(\mathbf I) \ar[r]^-H \ar@/_4ex/[rr]_-{F\otimes \mathrm F_\K(u\mapsto 1_\K) } &
   \tilde B \ar[r]^-{\exists!\,Q}&
    B
}
\end{equation}
By Lemma~\ref{lemma:pushout_J_can}, $H$ is injective on the objects of the form $(x,1)$, hence $J$ is injective on objects and thus (by the inclusion we have already proved) is a cofibration. Also by Lemma~\ref{lemma:pushout_J_can} $G$ is a $\K$-linear equivalence, implying that $Q$ is. 
Since $Q$ is obviously surjective on objects, it is actually a trivial fibration. This proves~(iv).

Now let $F\colon A\to B$ be any cofibration and factor it as $A\to \tilde B \stackrel{\sim}{\to} B$, according to~(iv). In particular $A\to \tilde B$ is injective on objects.
Since $F$ is a cofibration and $\tilde B \stackrel{\sim}{\to} B$ is a trivial fibration, there exists a lifting in the following square
\[
\xymatrix{
A \ar[d]_F \ar[r] & {\tilde B} \ar[d]^{\sim} \\
B \ar@{..>}[ur] \ar@{=}[r] & B
}
\]
which implies that $F$ must also be injective on objects. This shows $\Cof \subseteq \mathcal C$ and thus $\Cof=\mathcal C$, as claimed.
Finally, the easy verification of (ii) proceeds exactly as in the proof of \cite{ivo:unitary}*{Proposition~4.16}, using the characterization of cofibrations we have just proved.
\end{proof}

\begin{remark}
\label{rem:cylinder}
Applying (iv) of Theorem~\ref{thm:can} we obtain the following cylinder
\begin{equation}\label{eq:cylinder}
\xymatrix{
*+<3.0ex>{\K \sqcup \K} \ar[rr]^-{(1_\K,1_\K)} \ar@{>->}[dr]_{(J_1,J_2)} && \K \\
& \mathrm F_\K \mathbf I \ar@{->>}[ur]_Q^\sim & 
}
\end{equation}
on $\K \in \Cat_\K$ (factor $F=(1_\K,1_\K)$). Here, $J_i$ is the unique $\K$-linear functor sending the unique object of $\K$ to $i\in \obj(\mathrm F_\K \mathbf I)$, and $Q$ the unique $\K$-linear functor sending the isomorphism $u$ to the identity. 
The corresponding canonical cylinder for any object $A\in \Cat_\K$ can then be obtained by tensoring \eqref{eq:cylinder} with~$A$.
\end{remark}

\begin{cor}
\label{cor:Ho(M_can)}
The category $\Ho(\cM_\can)$ is obtained from $\Cat_\K$ simply by taking as morphisms the isomorphism classes of $\K$-functors.
\end{cor}

\begin{proof}
Since every object is fibrant and cofibrant in the canonical model, we have a natural identification 
 $\Hom_{\Ho(\cM_\can)}(A,B)=\Hom_{\Cat_\K}(A,B)/_\sim$,  where the equivalence relation $\sim$ is the homotopy relation defined by the canonical cylinder objects of Remark~\ref{rem:cylinder}. Now it suffices to notice that, for any pair of parallel $\K$-functors $F_0,F_1\colon A\rightrightarrows B$, the homotopies $H\colon A\otimes \mathrm F_\K\mathbf I \to B $ from $F_0$ to $F_1$ are in bijection with the isomorphisms $F_0\simeq F_1$ of $\K$-functors.
\end{proof}

\subsection*{Morita model structure}
Given a $\K$-category~$A$, let $\mathsf P(A)$ be the full subcategory of $\Mod A$ consisting of those $A$-modules $M$ such that the represented functor $\Hom_{A}(M,-)\colon \Mod A\to \Mod \K$ commutes with arbitrary colimits. Consider also the full subcategory $\proj A\subseteq \Mod A$ of all finitely generated projective $A$-modules. As usual, an object $M\in \Mod A$ is \emph{projective} if $\Hom_A(M,-)$ is exact, and it is \emph{finitely generated} if there exists an epimorphism $F\to M$ with $F$ a finite coproduct of representable right $A$-modules.

\begin{lemma}
\label{lemma:saturation_vs_P}
We have an equality $\mathsf P(A)=\proj A$ and a natural $\K$-linear equivalence $A^\natural_\oplus \stackrel{\sim}{\to} \mathsf P(A)$.
\end{lemma}

\begin{proof}
This is well-know, but we reprove it for convenience.
Recall (e.g.\ from \cite{pareigis}*{\S4.11 Lemma~1}) that, in the Grothendieck category $\Mod A$, an object $M$ is finitely generated projective if and only if it is projective and the functor $\Hom_A(M,-)$ preserves arbitrary coproducts. Thus $M\in \proj A$ if and only if  $\Hom_A(M,-)$ preserves coproducts and cokernels. But since every colimit can be written as a cokernel of a map between two coproducts, this is the same as preserving arbitrary colimits. 
Thus $\mathsf P(A)=\proj A$, as claimed. By the Yoneda lemma, the representable modules generate the abelian category $\Mod A$ and are projective. Thus, on the one hand, since $\proj A$ is saturated the Yoneda embedding $h_A\colon A\to \Mod A$ induces a $\K$-linear functor $A^\natural_\oplus \to \proj A$, which is unique up to isomorphism.
On the other hand, if $P$ is projective then every epimorphism from a coproduct of representables onto $P$ must split; if $P$ is moreover finitely generated, then the splitting factors through a finite summand. Therefore $\proj A$ consists precisely of the retracts of finite sums of representables. Hence the functor $A^\natural_\oplus \to \proj A$ is  an equivalence. By uniqueness, it is also natural up to isomorphism. (Such naturality will suffice for all our purposes; for a stronger statement, one would have to choose canonical cokernels in $\Mod A$.)
\end{proof}

\begin{prop}
\label{prop:satu_equiv}
A $\K$-functor $F\colon A\to B$ between small $\K$-categories is a Morita equivalence (Definition~\ref{defi:Morita_eq}) if and only if $F_\oplus^\natural\colon A_\oplus^\natural\to B_\oplus^\natural$ is a $\K$-equivalence. In other words, $F$ is a Morita equivalence if and only if $F$ is fully faithful and the image of $\iota_B\circ F$ additively generates~$B^\natural_\oplus$.
\end{prop}

\begin{proof}
The equivalence between the two descriptions of $\K$-equivalence follows from Lemma~\ref{lemma:first_char}. To prove the first description, consider the following diagram of $\K$-linear functors:
\[
\xymatrix@R=13pt{
A \ar[d] \ar[r]^-F & B \ar[d] \\
\mathsf P(A) \ar[d] \ar[r]^{F'} & \mathsf P(B) \ar[d] \\
\Mod A \ar[r]^{F''} & \Mod B
}
\] 
The vertical arrows are the inclusions; $F'$ is the unique extension of $F$ commuting with direct sums and retracts, i.e., the functor identifying with $F^\natural_\oplus$ under the natural equivalence of the previous lemma; and $F''$ is the left Kan extension of $A\to \Mod B$ along $A\to \Mod A$, and also of $\mathsf P(A)\to \Mod B$ along $\mathsf P(A)\to \Mod A$ (see \cite{maclane}*{\S X.3}). Since the functors $A\to \mathsf P(A)$ and $\mathsf P(A)\to \Mod A$ are fully faithful, the diagram commutes up to isomorphism.
We claim that $F''$ is an equivalence if and only if $F'$ is; then the proposition will follow by  Lemma~\ref{lemma:saturation_vs_P}.
In one direction, if $F''$ is an equivalence then it must restrict to an equivalence $\mathsf P(A)\stackrel{\sim}{\to} \mathsf P(B)$, because the properties of being finitely generated and projective are categorical; hence $F'$ is an equivalence. In the other direction, if $F'$ is an equivalence then it has a quasi-inverse $G'$, which will induce its own left Kan extension $G''\colon \Mod B\to \Mod A$. The uniqueness property of left Kan extensions now shows that $F''$ and $G''$ are mutually quasi-inverse equivalences.
\end{proof}

\begin{cor}\label{cor:sat_Morita_eq}
If $D$ is a (not necessarily small) saturated $\K$-category, then the functor $\Fun_\K(-,D)$ sends Morita equivalences to $\K$-equivalences.
\end{cor}
\begin{proof}
Let $F\colon A\to B$ be a Morita equivalence. By Proposition~\ref{prop:satu_equiv}, $F$ is fully faithful and the image of $\iota_B \circ F$ additively generates $B^\natural_\oplus$; \emph{a fortiori}, the image of $F$ additively generates~$B$. Since $F$ is fully faithful, one checks easily that the $\K$-linear functor $\Fun_\K(F,D)\colon \Fun_\K(B,D)\to \Fun_\K(A,D)$ is fully faithful. Since $D$ is saturated and the image of $F$ additively generates~$B$, it follows from Lemma \ref{lemma:unique_extension} that $\Fun_\K(F,D)$ is essentially surjective. Therefore it is a $\K$-linear equivalence.
\end{proof}

\begin{defi}\label{defi:Mor_model}
Define the Morita model structure on $\Cat_\K$ to be $\cM_\Mor := L_S \cM_\can$, the left Bousfield localization of the canonical model structure of Theorem~\ref{thm:can} with respect to the set $S:=\{R_0, R_1, S_2\}$ consisting of the following three $\K$-linear functors:
 \begin{itemize}
 \item[(i)]

 $R_0\colon \emptyset \to 0$ is the unique functor from the initial to the final object, i.e., from the empty to the zero $\K$-category.
 \item[(ii)]
$R_1\colon E(1)\to R(1)$ is the universal addition of a retract. More precisely, let $E(1)$ be the $\K$-category generated by one object~$o$ equipped with an idempotent endomorphism $e=e^2\colon o\to o$, and let $R(1)$ be the $\K$-category generated by two objects $o$ and~$r$, two arrows $p\colon o\to r$ and $i\colon r\to o$, and the relation $pi=1_r$.
Then $R_1$ is the unique (fully faithful) functor sending $e$ to the idempotent~$ip$.

\item[(iii)]
$S_2\colon \K\sqcup \K \to S(2)$ is the universal addition of a direct sum. More precisely, $S(2)$ is generated by three objects $o_1,o_2$ and~$s$, arrows $i_k \colon o_k \to s$ and $p_k \colon s\to o_k$ ($k =1,2$), and relations $p_ki_k= 1_{o_k}$ ($k=1,2$) and $i_1p_1+ i_2p_2=1_s$.
Then $S_2$ is the unique (fully faithful) $\K$-linear functor $\K\sqcup \K\to S(2)$ sending the first copy of $\K$ to $o_1$ and the second copy to~$o_2$.
\end{itemize}
The left Bousfield localization, and therefore the Morita model structure, is well-defined because the canonical model $\cM_\can$ is left proper (since all objects are cofibrant) and because by item (iii) of Theorem \ref{thm:can} it is combinatorial; see \cites{barwick, beke:sheafifiable}. Moreover, $\cM_\Mor$ ihnerits the property of being combinatorial.
Since $\cM_\can$ is symmetric monoidal we can use here the \emph{$\cM_\can$-enriched} version of Bousfield localization (see~\cite{barwick}) rather than the more common simplicial version (the result is the same).
\end{defi}
Let us recall what this all means, in the situation at hand. An object $D\in \Cat_\K$ is \emph{$S$-local} if for every $(F\colon A\to B)\in S$, the induced map
\begin{equation}\label{eq:induced_S}
F^*=\Hom_{\Ho(\cM_\can)}(F,D) \colon\Hom_{\Ho(\cM_\can)}(B,D)\to \Hom_{\Ho(\cM_\can)}(A,D)
\end{equation}
is a bijection. 
A $\K$-linear functor $F\colon A\to B$ is an \emph{$S$-local equivalence} if, conversely, \eqref{eq:induced_S} is a bijection for every $S$-local object~$D$. By definition, the weak equivalences of $\cM_\Mor$ are precisely the $S$-local equivalences, the cofibrations are the same as those of $\cM_\can$, and the fibrations are determined as usual by the right lifting property with respect to the trivial cofibrations.
By the theory, the Morita fibrant objects (i.e.\ those objects that are fibrant for the Morita model) are precisely the $S$-local ones. 

\begin{notation}
By default, $\Ho(\Cat_\K)$ always refers to $\Ho(\cM_\Mor)$.
\end{notation}

\begin{lemma}
\label{lemma:Mor_fibrant}
A $\K$-category is Morita fibrant if and only if it is saturated (see \S\ref{sec:preliminaries}).
\end{lemma}

\begin{proof}
Consider the following three lifting problems for an object $D$ in $\Cat_\K$:
\begin{equation}\label{eq:liftings_S}
\xymatrix{ \emptyset \ar[d]_{R_0} \ar[d] \ar[r] & D \\ 0 \ar@{..>}[ur] }
\quad \quad
\xymatrix{ E(1) \ar[d]_{R_1} \ar[d] \ar[r]^-G & D \\ R(1) \ar@{..>}[ur] }
\quad \quad
\xymatrix{ \K\sqcup \K \ar[d]_{S_2} \ar[d] \ar[r]^-H & D \\ S(2) \ar@{..>}[ur] }
\end{equation}
Exactly as in the proof of \cite{CcatMorita}*{Proposition 4.24}, it is easy to verify from the definitions that the unique functor $\emptyset \to D$ lifts along $R_0$ if and only if $D$ has a zero object; 
that every $G$ as above lifts along $R_1$ if and only if every idempotent of $D$ splits; and that every $H$ as above lifts along $S_2$ if and only if any two objects of $D$ have a direct sum.
Hence, $D$ has the right lifting property with respect to the set $S$ precisely when it is saturated.
Now we must show that $D$ has the right lifting property with respect to $S$ if and only if it is $S$-local.
Since the three maps in $S$ are trivial fibration, of course $S$-locality (i.e.\ Morita fibrancy) implies the lifting propery. So it only remains to prove the converse. 

Let $D$ have the right lifting property with respect to each~$F\in S$. Note that, by the uniqueness of zero objects, retracts, and direct sums, the resulting liftings in \eqref{eq:liftings_S} are unique up to a canonical isomorphism of $\K$-functors. This implies that for any $(F\colon A\to B)\in S$ the induced $\K$-functor
\[
F^*=\Fun_\K(F,D)\colon \Fun_\K(B,D) \longrightarrow \Fun_\K(A,D)
\]
is essentially surjective. Similarly, $F^*$ is easily seen to be fully faithful so it is an equivalence. By considering isomorphism classes of objects, we deduce with Corollary \ref{cor:Ho(M_can)} that \eqref{eq:induced_S} is a bijection. Thus $D$ is $S$-local, as claimed.
\end{proof}

\begin{lemma}
\label{lemma:Mor_eq_eq}
The weak equivalences in the Morita model structure are precisely the Morita equivalences in the sense of Definition~\ref{defi:Morita_eq}.
\end{lemma}

\begin{proof}
Consider a $\K$-linear functor $F\colon A\to B$. 
It is a weak equivalence of the Morita model structure if and only if it is an $S$-local equivalence, i.e.\ by Corollary \ref{cor:Ho(M_can)} and Lemma~\ref{lemma:Mor_fibrant}, if and only if the induced functor $F^*\colon \Fun_\K(B,D)\to \Fun_\K(A,D)$ is an equivalence of categories for all saturated~$D$. Consider the following commutative diagram:
\[
\xymatrix{
\Fun_\K(B,D) \ar[r]^-{F^*} &
 \Fun_\K(A,D) \\
\Fun_\K(B^\natural_\oplus, D) \ar[u]^{(\iota_B)^*}_\simeq \ar[r]^-{(F^\natural_\oplus)^*} &
 \Fun_\K(A^\natural_\oplus, D) \ar[u]_{(\iota_A)^*}^\simeq
}
\]
By Corollary \ref{cor:sat_Morita_eq} and since $\iota_A$ and $\iota_B$ are Morita equivalences, $(\iota_A)^*$ and $(\iota_B)^*$ are equivalences.
If $F$ is a Morita equivalence then $\smash{F^\natural_\oplus}$ is an equivalence (by Proposition~\ref{prop:satu_equiv} again), hence so is~$\smash{(F^\natural_\oplus)^*}$, and it follows from the commutative diagram that $F^*$ is an equivalence. This shows that Morita equivalences are $S$-local equivalences. 
Together with Lemma~\ref{lemma:Mor_fibrant}, this also proves Corollary \ref{cor:Mor_fibrant_replacement} below. 

Conversely, assume that $F$ is an $S$-local equivalence; thus in the diagram $F^*$ is an equivalence for all saturated~$D$. Hence $(F^\natural_\oplus)^*$ is an equivalence too.
By Corollary \ref{cor:Mor_fibrant_replacement}, and since every object is Morita cofibrant, this is the same as saying that $\Hom_{\Ho(\Cat_\K)}(F^\natural_\oplus,C)$ is bijective for all $C\in \Cat_\K$. By Yoneda, $F^\natural_\oplus$ is invertible in $\Ho(\Cat_\K)$, i.e.\ it is an $S$-local equivalence. But its domain and codomain are Morita fibrant, hence $\smash{F^\natural_\oplus}$ is already a canonical weak equivalence, i.e.\ an equivalence of $\K$-categories. In other words, $F$ is a Morita equivalence.
\end{proof}

\begin{cor}
\label{cor:Mor_fibrant_replacement}
The natural embedding $\iota_A\colon A\to A^\natural_\oplus$ provides a functorial fibrant replacement for the Morita model structure. 
\qed
\end{cor}

We now obtain an explicit description of the Morita homotopy category.

\begin{prop}
\label{prop:Mor_homs}
For any two $A,B\in \Cat_\K$, there is a canonical bijection between maps $\varphi\colon A\to B$ in $ \Ho(\Cat_\K)$ and isomorphism classes of $\K$-linear functors $F\colon A\to B^\natural_\oplus$, obtained by sending the isomorphism class of $F$ to the equivalence class of the left fraction $\iota_B^{-1} \circ F $. 
The map $\varphi$ is invertible if and only $F$ is a Morita equivalence. 
If $\varphi\colon A\to B$ and $\psi\colon B\to C$ are represented by $F$ and $G$, respectively, then their composite $\psi \circ \varphi$ is represented by $\tilde G\circ F$, where $\tilde G\colon B^\natural_\oplus\to C^\natural_\oplus$ is the (up to isomorphism, unique) $\K$-linear extension of $G$ along $\iota_B\colon B\to B^\natural_\oplus$.
\end{prop}

\begin{proof}
By Corollary \ref{cor:Mor_fibrant_replacement}, and since every object is cofibrant, every map $\varphi\colon A\to B$ in the homotopy category is the equivalence class of some fraction $\smash{A \stackrel{F}{\to } B^\natural_\oplus \stackrel{\iota}{\leftarrow} B}$.
Two $\K$-linear functors $\smash{F,F'\colon A\to B^\natural_\oplus}$ represent the same map precisely when they are isomorphic, as claimed, by Corollary \ref{cor:Ho(M_can)}.
The other claims now follow immediately from the naturality of $\iota$ and the fact that saturation is idempotent up to equivalence
(cf.\ \cite{CcatMorita}*{Proposition~4.27}). 
\end{proof}

\begin{prop}\label{prop:monoidal}
The Morita model category, endowed with its closed symmetric monoidal structure (see \S\ref{sec:preliminaries}), is a symmetric monoidal model category.
It follows in particular that $\Ho(\Cat_\K)$ is a closed symmetric monoidal category. 
\end{prop}

\begin{remark}
\label{rem:Ho_tensor}
Note that, since every object is cofibrant, we do not need to derive the tensor product $-\otimes-$. However, we need to derive the internal Homs $\Fun_\K(-,-)$, and we can do this simply by saturating the target category.
\end{remark}

\begin{proof}
Let $F\colon A \to B$ be a $\K$-linear functor. We show first that $C\otimes F$ is a Morita equivalence whenever $F$ is a Morita equivalence and $C\in \Cat_\K$. Indeed, consider the commutative square in $\Cat_\K$
\[
\xymatrix{
\Fun_\K (C\otimes A, D) \ar[r]^-{\simeq} &
 \Fun_\K (A, \Fun_\K(C,D))  \\
\Fun_\K (C\otimes B, D) \ar[u]^{F^*} \ar[r]^-{\simeq} &
 \Fun_\K (B, \Fun_\K(C,D)) \ar[u]_{F^*}
}
\]
where $D$ is any saturated $\K$-category. Since $D$ is saturated so is $\Fun_\K(C,D)$; since $F$ is a Morita equivalence, the rightmost $F^*$ is then an equivalence. It follows that the leftmost $F^*$ is also one. Since $D$ is an arbitrary saturated $\K$-category, this proves that $C\otimes F$ is a Morita equivalence, as claimed (of course, for this argument we make implicit use of the characterizations of  Morita fibrant objects and Morita local equivalences of Lemmas \ref{lemma:Mor_fibrant} and \ref{lemma:Mor_eq_eq}). 
Now we can directly verify the definition of a symmetric monoidal model, precisely as in the proof of \cite{CcatMorita}*{Proposition 6.3}.
\end{proof}

\subsection*{Homotopy of $\K$-algebras}
Let $R$, $S$ and $T$ be $\K$-algebras, considered as objects in $\Cat_\K$. 
A bimodule ${}_RM_S$ which is finitely generated projective as a right $S$-module is the same data as a $\K$-linear functor $R\to \proj  S$. 
Two such bimodules ${}_RM_S$ and ${}_RM'_S$ are isomorphic as bimodules if and only if the corresponding $\K$-linear functors are isomorphic.
The $\K$-linear functor corresponding to the tensor product ${}_RM \otimes_SN_R$ identifies with the composition of $R\to \proj S$ with the canonical extension of $S \to \proj T$ along the embedding $S\to \proj S$. These facts combined with Proposition~\ref{prop:Mor_homs} furnish us with the following explicit description of the full subcategory of $\Ho(\Cat_\K)$ of $\K$-algebras, which (with a slight abuse of notation) we will denote by $\Ho(\Alg_\K)$. Namely, its objects are the $\K$-algebras. The Hom sets $ \Hom_{\Ho(\Alg_\K)}(R,S)$ are given by the isomorphism classes of the category $\mathrm{rep}(R,S)$ of those $R\text{-}S$-bimodules which are finitely generated projective as $S$-modules. Composition is given by 
\begin{eqnarray*}
\mathrm{Iso}\,\mathrm{rep}(S,T) \times \mathrm{Iso}\,\mathrm{rep}(R,S) \to \mathrm{Iso}\,\mathrm{rep}(R,T) && ([{}_SN_T],[{}_RM_S]) \mapsto [{}_RM\otimes_S N_T]
\end{eqnarray*}
and the tensor structure is induced by the tensor product of $\K$-algebras.
\section{Tensor categorification and base-change}\label{sec:tensor-categorification}
In this section we $\otimes$-categorify the Brauer group; see Corollary~\ref{cor:identification}. Then, we establish the base-change functoriality of this $\otimes$-categorification; see Corollary~\ref{cor:base-change}.
\subsection*{$\otimes$-categorification}
Let $\K$ be a commutative ring. Recall that Azumaya $\K$-algebras can be defined as those $\K$-algebras $A$ for which there exists a $\K$-algebra~$B$, a faithful finitely generated projective $\K$-module $P$, and an isomorphism $A\otimes B\simeq \End_\K(P)$ of $\K$-algebras (see \cite{knus-ojanguren}*{Theorem 5.1}). 
Recall also (e.g.\ from \cite{wisbauer-foundations}*{18.11}) that a finitely generated projective $\K$-module $P\in \proj \K$ is faithful precisely when it is a generator of $\Mod \K$, or equivalently, when it additively generates $\proj \K$.

\begin{prop}\label{prop:Azumaya}
The $\otimes$-invertible objects in $\Ho(\Cat_\K)$ are precisely the $\K$-cat\-e\-go\-ries which are Morita equivalent to Azumaya $\K$-algebras.
\end{prop}
\begin{proof}
If $P\in \proj \K \simeq \K^\natural_\oplus$ is an additive generator then $\Mod \End_\K(P)\simeq \Mod \K$ by Corollary~\ref{cor:sat_Morita_eq}, and it follows immediately from the characterization recalled above that Azumaya $\K$-algebras are $\otimes$-invertible objects in $\Ho(\Cat_\K)$. 

Now consider two $\K$-categories $A$ and $B$ such that $A\otimes B\simeq \K$ in $\Ho(\Cat_\K)$. 
We may assume, without loss of generality, that $A$ and $B$ are saturated.
By Proposition \ref{prop:Mor_homs} this isomorphism can be realized by a Morita equivalence $F\colon A\otimes B \to \K^\natural_\oplus$. 
In particular there exist finitely many $(x_i,y_i) \in \obj A\times \obj B$ such that $\bigoplus_iF(x_i,y_i)$ is an additive generator for $\smash{\K_\oplus^\natural \simeq \proj \K}$. By setting $x:=\bigoplus_i x_i$ and $y:=\bigoplus_i y_i$, the object $P:=F(x,y)$ is then also an additive generator of $\smash{\K^\natural_\oplus}$, because it displays $\bigoplus_iF(x_i,y_i)$ as a retract; indeed we have $F(x,y)=\bigoplus_{i,j}F(x_i,y_j)$.
It follows that the $\K$-algebra $\End_\K(P)$ is Morita equivalent to~$\K$ and since $F$ is fully-faithful we obtain an isomorphism $A(x,x)\otimes B(y,y)\stackrel{\sim}{\to} \End_\K(P)$. Therefore $A(x,x)$ is an Azumaya $\K$-algebra. 
Let us now show that $A(x,x)$ and $A$ are isomorphic in $\Ho(\Cat_\K)$, \ie\ that they are Morita equivalent. Consider the composite functor
\begin{equation*}
\xymatrix{
A(x,x)  \otimes B \ar[r]^-{I\otimes B} & A\otimes B\ar[r]_-\sim^-{F} & \K_\oplus^\natural
}
\end{equation*}
where $I\colon A(x,x) \to A$ denotes the inclusion.
By construction, the composite is fully faithful and its image contains the additive generator $P$ of~$\smash{\K_\oplus^\natural}$. As a consequence, it is a Morita equivalence. By the 2-out-of-3 property, one concludes that $I\otimes B$ is also a Morita equivalence. Since $-\otimes B$ is an endo-equivalence of $\Ho(\Cat_\K)$, the inclusion $I\colon A(x,x)\to A$ must also be a Morita equivalence, thus proving that $A(x,x)$ and $A$ are isomorphic in $\Ho(\Cat_\K)$. 
This concludes the proof.
\end{proof}

\begin{cor}
\label{cor:identification}
We obtain  a natural isomorphism $\mathrm{Pic}(\Ho(\Cat_\K)^\otimes) \simeq \Br(\K)$. 
\qed
\end{cor}

\begin{remark}
As hinted in the introduction, this description of the Brauer group is well-known. For instance, the category $\Ho(\Cat_\K)$ appears briefly under the name $Cm_k$ in \cite{toen:azumaya}*{p.\,597}, described as in Proposition~\ref{prop:Mor_homs} (but without the Morita model structure and over a base field~$k$).
\end{remark}

\subsection*{Base-change}
Let $\K \to \mathbb{L}$ be a homomorphism of commutative rings. By base-change one obtains a $\otimes$-functor
\begin{eqnarray}\label{eq:base-change}
-\otimes_\K \mathbb{L}: \Cat_\K \too \Cat_{\mathbb{L}} && A \mapsto A\otimes_\K \mathbb{L}\,.
\end{eqnarray}

\begin{lemma}\label{lem:Mor}
The functor \eqref{eq:base-change} preserves Morita equivalences.
\end{lemma}

\begin{proof}
Note first that, up to equivalence, the functor $-\otimes_\K \mathbb{L}$ commutes with the additive hull $(-)_\oplus$ but not with the idempotent completion $(-)^\natural$, as in general we only get a fully faithful inclusion $A^\natural \otimes_\K\mathbb{L} \subseteq (A \otimes_\K\mathbb{L})^\natural$. In order to prove this lemma we make use of the characterization of Proposition~\ref{prop:satu_equiv}.
Let $F\colon A\to B$ be a Morita equivalence in $\Cat_\K$.
Thus $F$ is fully faithful, and the full image of $F$ additively generates $B^\natural_\oplus$. 
It follows that $F\otimes_\K\mathbb{L}$ is fully faithful, and that we have a fully faithful embedding
\[
(B\otimes_\K\mathbb{L})_\oplus 
= B_\oplus \otimes_\K\mathbb{L}
\subseteq (\Image(F)_\oplus^\natural )\otimes_\K\mathbb{L} 
\subseteq  (\Image(F\otimes_\K\mathbb{L} ))_\oplus^\natural \,,
\]
showing that the full image of $F\otimes_\K\mathbb{L}$ additively generates $(B\otimes_\K\mathbb{L})_\oplus^\natural$. This allows us to conclude that the functor $F\otimes_\K\mathbb{L}\colon A\otimes_\K\mathbb{L} \to B\otimes_\K\mathbb{L}$ is a Morita equivalence.
\end{proof}

\begin{cor}\label{cor:base-change}
The base-change functor \eqref{eq:base-change} induces a well-defined $\otimes$-functor 
\[-\otimes_\K \mathbb{L}\colon \Ho(\Cat_\K)^\otimes \to \Ho(\Cat_{\mathbb{L}})^\otimes\]
 such that $\mathrm{Pic}(-\otimes_\K \mathbb{L})\simeq r_{\mathbb L/\K}$ under the identification of Corollary~\ref{cor:identification}.
 \qed
\end{cor}

Thus $\Pic$ recovers the covariant functoriality of the Brauer group. 

\section{Corestriction}\label{sec:corestriction}
In this section we build a corestriction $\otimes$-functor in the setting of (small) $\K$-categories. Then we verify its compatibility with the $\otimes$-categorification of \S\ref{sec:tensor-categorification}; see Proposition~\ref{prop:descent}. We assume that $\K$ is a field and that $\mathbb{L}$ is a finite Galois extension of $\K$ with Galois group $G:=\mathrm{Gal}(\mathbb{L}/\K)$ and degree $n:= [\mathbb{L}:\K]$. 
We start by recalling some definitions and results from Galois descent theory; consult \cite{draxl}*{\S6} or \cite{knus-ojanguren}*{II.\S5} for further details.

\begin{defi}\label{def:skew1}
An \emph{$\mathbb L/\K$-Galois module} is an $\mathbb{L}$-vector space $W$ endowed with a left $G$-action which is skew-linear, in the sense that $\sigma(x) \sigma(w)=\sigma(xw)$ for every $x\in \mathbb{L}, w \in W$ and $\sigma \in G$. 
Denote by $\GalMod \mathbb L/\K$ the category of $\mathbb L/\K$-Galois modules and $G$-equivariant $\K$-linear maps.
\end{defi}

\begin{prop}[{\cite{draxl}*{I.\S6 Theorem~1}}]
\label{prop:Speiser1}
Under the above hypotheses and notations, we have a natural $G$-equivariant $\mathbb{L}$-linear isomorphism
\begin{eqnarray}\label{eq:can}
\mathbb{L} \otimes_\K W^G \stackrel{\sim}{\too} W && \ell \otimes w \mapsto \ell w
\end{eqnarray}
for every $W\in \GalMod \mathbb L/\K$. Moreover, these isomorphisms form the counit of the following equivalence of categories:
\begin{equation*}
\xymatrix{
\GalMod \mathbb L/\K \ar@/^2ex/[d]^{(-)^G}_{\simeq\;} \\ 
\Mod \K \ar@/^2ex/[u]^{\mathbb{L} \otimes_\K-} \,.
}
\end{equation*}
\end{prop}

\begin{defi}\label{defi:skewaction1}
If $V$ is an $\mathbb L$-vector space, denote by ${}^\sigma V$ the $\mathbb{L}$-vector space that coincides with $V$ as a group and whose $\mathbb{L}$-action is given by $x \cdot v := \sigma^{-1}(x) v$ for $x \in \mathbb{L}$ and $v \in V$. 
Then $\smash{\bigotimes_{\sigma \in G} {}^\sigma V }$, where the tensor product is taken over $\mathbb{L}$, can be endowed with the skew-linear $G$-action $\tau(\otimes_{\sigma \in G} v_\sigma):= \otimes_{\sigma \in G} v_{\tau^{-1} \sigma}$, for $\tau\in G$. 
By taking $G$-invariants, we obtain in this way a well-defined functor
\begin{eqnarray}\label{eq:co-restriction}
\Cor_{\mathbb L/\K}\colon \Mod \mathbb{L} \too \Mod \K && V \mapsto \left(\textstyle\bigotimes_{\sigma \in G} {}^\sigma V \right)^G\,.
\end{eqnarray}
\end{defi}
Let us now describe some properties of this functor.

\begin{prop}\label{prop:monoidal1}
The functor \eqref{eq:co-restriction} is symmetric monoidal.
\end{prop}
\begin{proof}
The canonical $\K$-linear embedding
\begin{eqnarray*}\label{eq:can1}
\K \to \left(\textstyle\bigotimes_{\sigma \in G}{}^\sigma \mathbb{L} \right)^G  & 1 \mapsto 1 \otimes \cdots \otimes 1
\end{eqnarray*}
and the natural $\K$-linear homomorphisms ($V,W\in \Mod \mathbb L$)
\begin{equation*}\label{eq:can2}
\left(\textstyle\bigotimes_{\sigma \in G}{}^\sigma V \right)^G \otimes_\K \left(\textstyle\bigotimes_{\sigma \in G}{}^\sigma W \right)^G \to \left(\textstyle\bigotimes_{\sigma \in G}{}^\sigma (V\otimes_{\mathbb{L}}W) \right)^G\,,
\end{equation*}
sending $(\otimes_{\sigma \in G} v_\sigma) \otimes(\otimes_{\sigma \in G} w_\sigma)$ to $\otimes_{\sigma \in G} (v_\sigma \otimes w_\sigma)$, are easily seen to equip \eqref{eq:co-restriction} with a lax symmetric monoidal structure. 
 To see that they are invertible, by Proposition~\ref{prop:Speiser1} it suffices to show that their images under $\mathbb{L} \otimes_\K - \colon \Mod \mathbb K \to \GalMod \mathbb L/\K$ are invertible. Under the identification \eqref{eq:can} these images correspond to the evident isomorphisms $\mathbb{L} \simeq \mathbb{L} \otimes_\K \K \stackrel{\sim}{\rightarrow} \textstyle\bigotimes_{\sigma \in G}{}^\sigma \mathbb{L} \simeq \mathbb{L}$ and 
 $$(\textstyle\bigotimes_{\sigma \in G}{}^\sigma V )\otimes_{\mathbb{L}} (\textstyle\bigotimes_{\sigma \in G}{}^\sigma W ) \stackrel{\sim}{\to} \textstyle\bigotimes_{\sigma \in G}{}^\sigma (V\otimes W)\,.$$ 
This achieves the proof.
\end{proof}
\begin{lemma}\label{lem:dimension}
There is a natural isomorphism of $\K$-vector spaces
\[
 \Cor_{\mathbb L/\K}(\textstyle\bigoplus^m_{i=1} V) \simeq \textstyle\bigoplus_{i=1}^{m^{|G|}} \Cor_{\mathbb L/\K}(V)
\]
for every $V \in \Mod \mathbb L$ and integer $m \geq 1$.
\end{lemma}
\begin{proof}
One sees easily that there is a natural isomorphism of $\mathbb L$-vector spaces
\begin{eqnarray}\label{eq:iso1}
\bigotimes_{\sigma \in G} {}^\sigma(\bigoplus^m_{i=1} V) \stackrel{\sim}{\too} \bigoplus_{f:G \to \{1, \ldots, m\}} \bigotimes_{\sigma \in G} {}^\sigma V && \otimes_\sigma(v_{\sigma,i})_i \mapsto (\otimes_\sigma v_{\sigma, f(\sigma)})_f \,.
\end{eqnarray}
If one equips the left hand side with the skew-linear action of Definition~\ref{defi:skewaction1}, then \eqref{eq:iso1} becomes $G$-equivariant provided we equip the right hand side with the following skew-linear action: $\tau \cdot (\otimes_\sigma w_\sigma^f)_f := (\otimes_\sigma w_{\tau^{-1} \sigma}^{\tau^{-1}f})_f$, $\tau \in G$, where $\tau^{-1}f\colon G \to \{1,\ldots, m\}$ is the function $(\tau^{-1} f)(t):=f(\tau t)$, $t\in G$. There is also a natural $G$-equivariant isomorphism
\begin{eqnarray}\label{eq:iso2}
\bigoplus_{f:G \to \{1, \ldots, m\}} \bigotimes_{\sigma \in G} {}^\sigma V \stackrel{\sim}{\too} \bigoplus_{f:G \to \{1, \ldots, m\}} \bigotimes_{\sigma \in G} {}^\sigma V && (\otimes_\sigma w_\sigma^f)_f \mapsto (\otimes_\sigma w_\sigma^{\sigma \star f})_f\,,
\end{eqnarray}
where $\sigma \star f \colon G \to \{1, \ldots, m\}$ is the function $(\sigma\star f)(t):=f(t\sigma)$, $t\in G$, and the right hand side is equipped with the diagonal $G$-action. By applying $(-)^G$ to the composite \eqref{eq:iso2} $\circ$ \eqref{eq:iso1}, and after fixing an ordering for the elements of $G$, we obtain the claimed natural isomorphism.
\end{proof}

\begin{remark} \label{rem:algebras_to_algebras}
It follows from Proposition \ref{prop:monoidal1} that $\Cor_{\mathbb L/\K}$ maps $\mathbb L$-algebras $S$ to $\K$-algebras $\Cor_{\mathbb L/\K}(S)$ and left (resp.\ right) $S$-actions to left (resp.\ right) $\Cor_{\mathbb L/\K(S)}$-actions.
\end{remark}

\begin{prop}\label{prop:tensorproduct}
Let $R$, $S$ and $T$ be $\mathbb L$-algebras and $M$ an $R\text{-}S$-bimodule which is finitely generated and projective as a right $S$-module. Then, there is a canonical isomorphism of $\Cor_{\mathbb L/\K}(R)\text{-}\Cor_{\mathbb L/\K}(T)$-bimodules
$$ \Cor_{\mathbb L/\K}(M)\otimes_{\Cor_{\mathbb L/\K}(S)} \Cor_{\mathbb L/\K}(N) \simeq \Cor_{\mathbb L/\K}(M\otimes_S N)$$
for every $S\text{-}T$-bimodule $N$.
\end{prop}
\begin{proof}
In order to simplify the exposition we will simply write $C$ instead of $\Cor_{\mathbb L/\K}$. Recall that $M \otimes_S N$ is defined as the coequalizer
\begin{equation}\label{eq:coeq23}
\xymatrix{
M \otimes S \otimes N \ar@<.8ex>[rr]^-{\mu_M \otimes \id_N} \ar@<-.8ex>[rr]_-{\id_M \otimes \mu_N} && M \otimes N \ar[r] & M\otimes_S N\,,
}
\end{equation}
where $\mu_M$ (resp.\ $\mu_N$) denotes the right (resp.\ left) action of $S$ on $M$ (resp. on $N$). Note that $M\otimes_SN$ comes equipped with an $R\text{-}T$-bimodule structure induced by the left action of $R$ on $M$ and by the right action of $T$ on $N$. Since by Proposition~\ref{prop:monoidal1} the functor $C$ is symmetric monoidal it is enough to show that the induced diagram
\begin{equation}\label{eq:diagram}
\xymatrix{
C(M \otimes S \otimes N) \ar@<.8ex>[rr]^-{C(\mu_M \otimes \id_N)} \ar@<-.8ex>[rr]_-{C(\id_M \otimes \mu_N)} && C(M \otimes N) \ar[r] & C(M \otimes_S N)
}
\end{equation}
is a coequalizer in $\Mod \K$. The proof is now divided into three different cases:

{\it Case 1:} Assume that $M \simeq S$. In this case our claim is clear since $S\otimes_SN\simeq N$.

{\it Case 2:} Assume that $M \simeq S^{\oplus m}$ for some integer $m > 1$. In this case the above coequalizer \eqref{eq:coeq23} identifies with 
$$
\xymatrix{
\bigoplus_{i=1}^m \big(S\otimes S \otimes N \ar@<.8ex>[rr]^-{\mu_S \otimes \id_N} \ar@<-.8ex>[rr]_-{\id_S \otimes \mu_N} && S \otimes N \ar[r]  & S\otimes_S N  \big)
}
$$
and hence by Lemma~\ref{lem:dimension} the diagram \eqref{eq:diagram} identifies with the following direct sum of diagrams
\begin{equation}\label{eq:diag-big}
\xymatrix{
\bigoplus_{i=1}^{m^{|G|}}\big(C(S\otimes S\otimes N) \ar@<.8ex>[rr]^-{C(\mu_S \otimes \id_N)} \ar@<-.8ex>[rr]_-{C(\id_S \otimes \mu_N)}&& C(S\otimes N) \ar[r] & C(M\otimes_SN)  \big)\,.
}
\end{equation}
Now, since by Case~1 the diagram inside the brackets is a coequalizer in $\Mod \K$ one deduces that \eqref{eq:diag-big} is also a coequalizer. This proves Case~2.

{\it Case 3:} In full generality now, we can assume that $M$ is a retract of $S^{\oplus m}$ for some integer $m \geq 1$. By definition, there exist maps of right $S$-modules
\begin{eqnarray*}
\iota_M\colon M \to S^{\oplus m} & & \pi_M\colon S^{\oplus m} \too M
\end{eqnarray*}
verifying the equality $\pi_M \circ \iota_M = \id_M$. This allows us to construct the following diagram in $\Mod \mathbb{L}$:
\begin{equation}\label{eq:diagram-big1}
\xymatrix{
M \otimes S \otimes N \ar@<.8ex>[rr]^-{\mu_M \otimes \id_N} \ar@<-.8ex>[rr]_-{\id_M \otimes \mu_N}  && M \otimes N \ar[r] & M \otimes_SN \\
S^{\oplus m} \otimes S \otimes N \ar[u]^-{\pi_M \otimes \id_S \otimes \id_N} \ar@<.8ex>[rr]^-{\mu_{S^{\oplus m}} \otimes \id_N} \ar@<-.8ex>[rr]_-{\id_{S^{\oplus m}} \otimes \mu_N}  && S^{\oplus m} \otimes N  \ar[u]_-{\pi_M \otimes \id_N} \ar[r] & S^{\oplus m}\otimes_S N \ar[u]_-{\pi} \\
M \otimes  S \otimes N \ar[u]^-{\iota_M \otimes \id_S \otimes \id_N} \ar@<.8ex>[rr]^-{\mu_M \otimes \id_N} \ar@<-.8ex>[rr]_-{\id_M \otimes \mu_N}  && M \otimes N \ar[u]_-{\iota_M \otimes \id_N} \ar[r] & M \otimes_SN \ar[u]_-{\iota} \,,
}
\end{equation}
where $\iota$ and $\pi$ are induced by the universal property of the coequalizer. Note that the three vertical compositions are identities. By applying the functor $C$ to \eqref{eq:diagram-big1} we obtain, thanks to Case~2, an analogous diagram where the middle row is an equalizer. The proof now follows from the general Lemma~\ref{lem:general} below.
\end{proof}
\begin{lemma}\label{lem:general}
Consider the following commutative diagram
$$
\xymatrix{
\overline{X} \ar@<.8ex>[r]^-{\overline{f}} \ar@<-.8ex>[r]_-{\overline{g}}  & \overline{Y} \ar[r]^-{\overline{h}} & Z \\
X \ar[u]^-{\pi_X} \ar@<.8ex>[r]^-{f} \ar@<-.8ex>[r]_-{g}  & Y \ar[u]_-{\pi_Y} \ar[r]^-h & Z \ar[u]_-{\pi_Z} \\
\overline{X} \ar[u]^-{\iota_X} \ar@<.8ex>[r]^-{\overline{f}} \ar@<-.8ex>[r]_-{\overline{g}}  & \overline{Y} \ar[u]_-{\iota_Y} \ar[r]^-{\overline{h}} & \overline{Z} \ar[u]_-{\iota_Z} 
}
$$
where the middle row is a coequalizer and each vertical composite is the identity. By the commutativity of the left hand side squares we mean that $\overline{f} \circ \pi_X = \pi_Y \circ f$, $\overline{g}\circ \pi_X= \pi_Y \circ g$, $f \circ \iota_X = \iota_Y \circ \overline{f}$, and $g \circ \iota_X = \iota_Y \circ \overline{g}$. Under these assumptions, the first (and last) row is also a coequalizer.
\end{lemma}
\begin{proof}
Let $t\colon \overline{Y}\to T$ be such that $t\overline f=f\overline g$. Then $t\pi_Yf=t\overline f\pi_X= t\overline g\pi_X= t\pi_Yg$, hence there exists a unique $t'\colon Z\to T$ such that $t'h=t\pi_Y$. Now set $t'':=t'\iota_Z$. Then $t''\overline h=t' \iota_Z \overline h = t' h \iota_Y= t\pi_Y \iota_Y =t$. Since $\overline h$ is an epimorphism (being a retract of an epimorphism), $t''$ is the unique such morphism $\overline Z\to T$. This proves that the first row is a coequalizer, as claimed.
\end{proof} 

In the next definition, we extend Remark~\ref{rem:algebras_to_algebras} from algebras to categories.

\begin{defi}\label{defi:Cor}
The $\otimes$-functor \eqref{eq:co-restriction} induces, by ``base-change'' of enriched categories, a well-defined $\otimes$-functor 
\begin{equation}\label{eq:corestriction}
C_{\mathbb L/\K} \colon \Cat_{\mathbb{L}} \to \Cat_\K\,,
\end{equation}
 as follows. 
For $A\in \Cat_\mathbb L$, one obtains the $\K$-category $C_{\mathbb L/\K}(A)$ with the same object set as~$A$, and with composition maps
$\Cor\, A(y,z)\otimes_\K \Cor\, A(x,y) \simeq \Cor (A(y,z)\otimes_\mathbb L A(x,y)) \to \Cor\,A(x,z)$
and identities $\K\simeq \Cor\,\mathbb L \to \Cor \,A(x,x)$, obtained by applying the functor $\Cor=\Cor_{\mathbb L/\K}$ to the compositions and identities of $A$ and by composing with its monoidal structure maps, in the evident way. If $F\colon A\to B$ is an $\mathbb L$-functor with components $F(x,y)$, then $C_{\mathbb L/\K}(F)$ is the $\K$-functor $C_{\mathbb L/\K}(A)\to C_{\mathbb L/\K}(B)$ with components $\Cor\,F(x,y)$. 
\end{defi}
Recall from \S\ref{sec:intro} that the notation $c_{\mathbb L/\K}\colon\Br(\mathbb L) \to \Br(\K)$ stands for the corestriction map associated to the Galois extension. 
Classically, this map is defined via the cohomological description of the Brauer group (see \cite{Serre}*{Ch.\,X\,\S5}), as the corestriction (or transfer) map on second Galois cohomology. 
By a result of Riehm \cite{Riehm}*{Theorem~11}, the corestriction map can also be described by a functorial construction on Azumaya (i.e.\ central simple) algebras which, upon direct inspection (cf.\ \cite{draxl}*{\S8-9} for details), coincides with the functor $C_{\mathbb{L}/\K}$ in~\eqref{eq:corestriction}. 
Hence $[C_{\mathbb L/\K} (A)]= c_{\mathbb L/\K}(A)$ in $\Br(\K)$ for each Azumaya $\mathbb L$-algebra~$A$. This proves the second part of Proposition~\ref{prop:descent} below. We note in passing that, in Riehm's article, the fact that (his version of) the functor $C_{\mathbb{L}/\K}$ descends to the Brauer groups is deduced indirectly from  \cite{Riehm}*{Theorem~11}, while we have chosen to verify this directly within this section (cf.\ the end of the proof of \emph{loc.\,cit.}).  
\begin{prop}\label{prop:descent}
The corestriction functor \eqref{eq:corestriction} induces a well-defined $\otimes$-functor $C_{\mathbb{L}/\K}\colon \Ho(\Cat_{\mathbb{L}})^\otimes \to \Ho(\Cat_\K)^\otimes$ such that $\mathrm{Pic}(C_{\mathbb{L}/\K})\simeq c_{\mathbb L/\K}$ under the identification of Corollary~\ref{cor:identification}.
\end{prop}
\begin{proof}
It remains only to verify the first part, namely, that $C_{\mathbb L/\K}$ is well-defined at the level of the Morita homotopy categories. As proved in Lemma~\ref{lemma:alg} below, the functor~\eqref{eq:co-restriction} induces a well-defined $\otimes$-functor $C_{\mathbb L/\K}\colon \Ho(\Alg_{\mathbb{L}}) \to \Ho(\Alg_\K)$. Hence, by restricting attention to $\otimes$-invertible objects, we obtain a diagram
\[
\xymatrix{
\Ho(\Alg_\mathbb L)^\otimes \ar[r]^-{C_{\mathbb L/\K}} \ar[d]_\simeq & \Ho(\Alg_\K)^\otimes \ar[d]^\simeq \\
\Ho(\Cat_\mathbb L)^\otimes \ar@{..>}[r] & \Ho(\Cat_\K)^\otimes\,,
}
\]
where the vertical arrows are the symmetric monoidal fully faithful inclusions. Since these latter functors are equivalences by Proposition~\ref{prop:Azumaya}, we conclude that $C_{\mathbb L/\K}$ has an (essentially) unique extension to a symmetric monoidal functor 
$$C_{\mathbb L/\K}\colon \Ho(\Cat_\mathbb L)^\otimes \to \Ho(\Cat_\K)^\otimes$$
as claimed. This ends the proof of Proposition \ref{prop:descent}.
\end{proof}

\begin{lemma} \label{lemma:alg}
The functor \eqref{eq:co-restriction} induces a $\otimes$-functor $C_{\mathbb L/\K} \colon \Ho(\Alg_{\mathbb{L}}) \to \Ho(\Alg_\K)$.
\end{lemma}
\begin{proof} 
In order to simplify the exposition we will simply write $C$ instead of $\Cor_{\mathbb L/\K}$. Being monoidal, as we have already remarked $C$ maps $\mathbb L$-algebras to $\K$-algebras and bimodules to bimodules. Let $R$ and $S$ be two $\mathbb{L}$-algebras and ${}_RM_S$ a $R\text{-}S$-bimodule which is finitely generated and projective as a right $S$-module. 
Consider the description of $\Ho(\Alg)$ given at the end of Section~\ref{sec:can} in terms of bimodules and tensor products. 
If ${}_SN_T$ is a $S\text{-}T$-bimodule which is finitely generated and projective as a right $T$-module, we see thanks to Proposition~\ref{prop:tensorproduct} that $C$ preserves their composition: $[C(M \otimes_S N)]= [C(M)\otimes_{C(S)} C(N)]$. 
Similarly, it preserves identities: $[C({}_SS_S)]= [{}_{C(S)}C(S)_{C(S)}]$.
It remains only to show that the $C(R)\text{-}C(S)$-bimodule $C({}_RM_S)$ obtained is also finitely generated and projective as a right $C(S)$-module.
From Lemma~\ref{lemma:saturation_vs_P} one knows that $M_S$ belongs to~$\smash{S^\natural_\oplus}$. By functoriality of $C$ idempotents are mapped to idempotents and so one can assume without loss of generality that $M_S$ belongs to~$S_\oplus$. It remains then to show that $C(M)_{C(S)}$ belongs to $C(S)_\oplus$. (Note that this is not automatic since $C$ is {\em not} additive.) 
By combining the canonical isomorphism $C(\mathbb L)\simeq \K$ with Lemma~\ref{lem:dimension} we obtain, for every integer $m\geq 1$, an isomorphism
\begin{equation} \label{eq:concludes1}
C(\mathbb{L}^{\oplus m}) \simeq \textstyle\bigoplus_{i=1}^{m^{|G|}}\K\,.
\end{equation}
Hence, since $M_S \simeq S^{\oplus m}$ the following isomorphisms of right $C(S)$-modules hold:
\begin{align*}
C(M_S) 
& \;\simeq\;  C(S^{\oplus m}) \\
& \;\simeq\;  C(\mathbb{L}^{\oplus m} \otimes_\mathbb L S)\\
& \;\simeq\;  C(\mathbb{L}^{\oplus m}) \otimes_{\mathbb{K}} C(S) 
 & \textrm{by Prop.\,\ref{prop:monoidal1} }\\
& \;\simeq\;  (\textstyle\bigoplus_{i=1}^{m^{|G|}} \K) \otimes_\K C(S) & \textrm{by \eqref{eq:concludes1}} \\
& \;\simeq\;  \textstyle\bigoplus_{i=1}^{ m^{|G|}} C (S)\,.
\end{align*}
This shows that $C(M_S)$ belongs to $C(S)_\oplus$ and hence we conclude that $C$ descends to a well-defined functor $C_{\mathbb L/\K}\colon \Ho(\Alg_{\mathbb{L}}) \to \Ho(\Alg_\K)$.
Finally, the fact that this functor is symmetric monoidal is inherited from the corresponding property of~\eqref{eq:co-restriction}.
\end{proof}

The proof of Theorem~\ref{thm:main2} is now immediately obtained by combining  Proposition~\ref{prop:descent} with Corollaries~\ref{cor:identification} and \ref{cor:base-change}.


\begin{bibdiv}
\begin{biblist}


\bib{Ant-Gep}{article}{
   author={Antieau, B.},
   author={Gepner, D.},
   title={Brauer groups and \'etale cohomology in derived algebraic geometry},
   journal={arXiv 1210.0290},
   volume={},
   date={2012},
   number={},
   pages={},
}

\bib{AG}{article}{
   author={Auslander, M.},
   author={Goldman, O.},
   title={The Brauer group of a commutative ring},
   journal={Trans. Amer. math. Soc.},
   volume={97},
   date={1960},
   pages={367--409},
}


\bib{baez:week209}{article}{
   author={Baez, John},
   title={This Week's Finds in Mathematical Physics (Week 209)},
   journal={},
   volume={},
   date={2004},
   number={},
   pages={},
   eprint={http://math.ucr.edu/home/baez/week209.html}
}

\bib{barwick}{article}{
   author={Barwick, Clark},
   title={On (enriched) left Bousfield localizations of model categories},
   journal={preprint},
   volume={},
   date={2007},
   number={},
   pages={},
   issn={},
   review={ arXiv:0708.2067v2},
}

\bib{beke:sheafifiable}{article}{
   author={Beke, Tibor},
   title={Sheafifiable homotopy model categories},
   journal={Math. Proc. Cambridge Philos. Soc.},
   volume={129},
   date={2000},
   number={3},
   pages={447--475},
}

\bib{Berger-Moerdijk}{article}{
   author={Berger, C.},
   author={Moerdijk, I.},
   title={On the homotopy theory of enriched categories},
   journal={Q. J. Math.},
   volume={64},
   date={2013},
   number={3},
   pages={805--846},
}

\bib{B-Vitale}{article}{
   author={Borceux, F.},
   author={Vitale, E.},
   title={Azumaya categories},
   journal={Appl. Categ. Structures},
   volume={10},
   date={2002},
   number={5},
   pages={449--467},
}

\bib{ivo:unitary}{article}{
   author={Dell'Ambrogio, Ivo},
   title={The unitary symmetric monoidal model category of small $\rm C^*$-categories},
   journal={Homology Homotopy Appl.},
   volume={14},
   date={2012},
   number={2},
   pages={101--127},
}

\bib{CcatMorita}{article}{
   author={Dell'Ambrogio, I.},
   author={Tabuada, G.},
   title={Morita homotopy theory of $C^*$-categories},
   journal={J. Algebra},
   volume={398},
   date={2014},
   pages={162--199},
}

\bib{draxl}{book}{
   author={Draxl, P. K.},
   title={Skew fields},
   series={London Mathematical Society Lecture Note Series},
   volume={81},
   publisher={Cambridge University Press},
   place={Cambridge},
   date={1983},
   pages={ix+182},
}

\bib{Duskin}{article}{
   author={Duskin, John W.},
   title={The Azumaya complex of a commutative ring. Categorical algebra and its applications (Louvain-La-Neuve, 1987)},
   journal={Lecture Notes in Math.},
   volume={1348},
   date={1988},
   number={},
   pages={107-117},
   issn={},
}

\bib{hirschhorn}{book}{
   author={Hirschhorn, Philip S.},
   title={Model categories and their localizations},
   series={Mathematical Surveys and Monographs},
   volume={99},
   publisher={American Mathematical Society},
   place={Providence, RI},
   date={2003},
   pages={xvi+457},
}

\bib{hovey:model}{book}{
   author={Hovey, Mark},
   title={Model categories},
   series={Mathematical Surveys and Monographs},
   volume={63},
   publisher={American Mathematical Society},
   place={Providence, RI},
   date={1999},
   pages={xii+209},
}

\bib{Johnson}{article}{
   author={Niles Johnson},
   title={Azumaya objects in triangulated bicategories},
   journal={J. Homotopy Relat. Struct.},
   volume={},
   date={2013},
   number={},
   pages={},
   doi={10.1007/s40062-013-0035-6},
}

\bib{keller:dg}{article}{
   author={Keller, Bernhard},
   title={On differential graded categories},
   conference={
      title={International Congress of Mathematicians. Vol. II},
   },
   book={
      publisher={Eur. Math. Soc., Z\"urich},
   },
   date={2006},
   pages={151--190},
}

\bib{kelly-lack:VCAT}{article}{
   author={Kelly, G. M.},
   author={Lack, S.},
   title={$\scr V$-Cat is locally presentable or locally bounded if $\scr V$
   is so},
   journal={Theory Appl. Categ.},
   volume={8},
   date={2001},
   pages={555--575},
}

\bib{knus-ojanguren}{book}{
   author={Knus, M.-A.},
   author={Ojanguren, M.},
   title={Th\'eorie de la descente et alg\`ebres d'Azumaya},
   language={French},
   series={Lecture Notes in Mathematics, Vol. 389},
   publisher={Springer-Verlag},
   place={Berlin},
   date={1974},
}

\bib{Lurie}{book}{
   author={Lurie, Jacob},
   title={Higher topos theory},
   series={Annals of Mathematics Studies},
   volume={170},
   publisher={Princeton University Press},
   place={Princeton, NJ},
   date={2009},
   pages={xviii+925},
}

\bib{maclane}{book}{
   author={Mac Lane, Saunders},
   title={Categories for the working mathematician},
   series={Graduate Texts in Mathematics},
   volume={5},
   edition={2},
   publisher={Springer-Verlag},
   place={New York},
   date={1998},
   pages={xii+314},
}

\bib{M-Vitale}{article}{
   author={Moens, M.-A.},
    author={Vitale, E. M.},
   title={Groupoids and the Brauer group},
   journal={Cahiers Topologie GŽom. DiffŽrentielle CatŽg.},
   volume={41},
   date={2000},
   number={4},
   pages={305--313},
}

\bib{Morita}{article}{
   author={Morita, K.},
   title={Duality for modules and its applications to the theory of rings with minimum condition},
   journal={Sci. Rep. Tokyo KyoiKu Daigakn Sect. A},
   volume={6},
   date={1958},
   pages={83--142},
}

\bib{pareigis}{book}{
   author={Pareigis, Bodo},
   title={Categories and functors},
   series={Translated from the German. Pure and Applied Mathematics, Vol.
   39},
   publisher={Academic Press},
   place={New York},
   date={1970},
   pages={viii+268},
}

\bib{rezk:folk}{article}{
   author={Rezk, Charles},
   title={A model category for categories},
   conference={
      title={unpublished},
   },
   book={
      series={},
      volume={},
      publisher={},
      place={},
   },
   date={1996},
   eprint={http://www.math.uiuc.edu/~rezk/papers.html},
   pages={},
   review={},
}

\bib{Riehm}{article}{
   author={Riehm, I. A.},
   title={The corestriction of algebraic structures},
   journal={Invent. Math.},
   volume={11},
   date={1970},
   pages={73--98},
}

\bib{Serre}{book}{
   author={Serre, Jean-Pierre},
   title={Corps locaux},
   language={French},
   series={Publications de l'Institut de Math\'ematique de l'Universit\'e de
   Nancago, VIII},
   publisher={Actualit\'es Sci. Indust., No. 1296. Hermann, Paris},
   date={1962},
   pages={243},
}

\bib{toen:derived_Morita}{article}{
   author={To{\"e}n, Bertrand},
   title={The homotopy theory of $dg$-categories and derived Morita theory},
   journal={Invent. Math.},
   volume={167},
   date={2007},
   number={3},
   pages={615--667},
}

\bib{toen:azumaya}{article}{
   author={To{\"e}n, Bertrand},
   title={Derived Azumaya algebras and generators for twisted derived
   categories},
   journal={Invent. Math.},
   volume={189},
   date={2012},
   number={3},
   pages={581--652},
}

\bib{wisbauer-foundations}{book}{
   author={Wisbauer, Robert},
   title={Foundations of module and ring theory},
   series={Algebra, Logic and Applications},
   volume={3},
   edition={Revised and translated from the 1988 German edition},
   note={A handbook for study and research},
   publisher={Gordon and Breach Science Publishers},
   place={Philadelphia, PA},
   date={1991},
   pages={xii+606},
}


\end{biblist}
\end{bibdiv}
\end{document}